\documentclass[]{amsart}
\usepackage[left=1.5in, top=1in, bottom=1in, right=1.5in, asymmetric]{geometry}
\usepackage{graphicx}
\usepackage{amssymb}
\usepackage{amsfonts}
\usepackage{bold-extra}
\usepackage{epstopdf}
\usepackage{amsmath}
\usepackage{amssymb}
\usepackage{epsfig}
\usepackage{graphicx}
\usepackage{changebar}
\usepackage{centernot}
\usepackage{rotating}
\usepackage{amsthm}
\usepackage{caption}
\usepackage{rotating}
\usepackage{setspace}
\usepackage{tikz}
\usepackage{subfigure}
\usepackage{esint}
\usepackage{cancel}
\usepackage{mathrsfs}
\usepackage{mathtools}
\usepackage{enumerate}
\usepackage{xcolor}
\allowdisplaybreaks

\newcommand{\R}{\mathbb{R}}

\newcommand{\Lip}{\text{Lip}}

\newcommand{\ve}{\varepsilon}
\newcommand{\vE}{\mc{E}}

\newcommand{\beq}{\begin{equation}}
\newcommand{\eeq}{\end{equation}}

\newcommand{\dx}{\partial}

\makeatletter
\DeclareRobustCommand\widecheck[1]{{\mathpalette\@widecheck{#1}}}
\def\@widecheck#1#2{%
	\setbox\z@\hbox{\m@th$#1#2$}%
	\setbox\tw@\hbox{\m@th$#1%
		\widehat{%
			\vrule\@width\z@\@height\ht\z@
			\vrule\@height\z@\@width\wd\z@}$}%
	\dp\tw@-\ht\z@
	\@tempdima\ht\z@ \advance\@tempdima2\ht\tw@ \divide\@tempdima\thr@@
	\setbox\tw@\hbox{%
		\raise\@tempdima\hbox{\scalebox{1}[-1]{\lower\@tempdima\box
				\tw@}}}%
	{\ooalign{\box\tw@ \cr \box\z@}}}
\makeatother

\def\mc{\mathcal}

\def\t{\text}

\newtheorem{Lem}{Lemma}

\newtheorem{Thm}[Lem]{Theorem}

\newtheorem{Def}{Definition}

\numberwithin{equation}{section}

\title[Weak Solutions of Thin Film Equations]{Nonnegative Weak Solutions of Thin Film Equations Related to Viscous Flows in Cylindrical Geometries}

\author[J.L. Marzuola]{Jeremy L. Marzuola}
\email{marzuola@math.unc.edu}
\address{Department of Mathematics, University of North Carolina at Chapel Hill \\ CB\#3250
  Phillips Hall \\ Chapel Hill, NC 27599}
  
    \author[S. Swygert]{Sterling Swygert}
\email{swygerts@live.unc.edu}
\address{Department of Mathematics, University of North Carolina at Chapel Hill \\ CB\#3250
  Phillips Hall \\ Chapel Hill, NC 27599}

    \author[R. Taranets]{Roman Taranets}
\email{taranets\_r@yahoo.com}
\address{Institute of Applied Mathematics and Mechanics of the NASU \\ 1, Dobrovol\'skogo Str., 84100, Sloviansk, Ukraine}

\begin{document}

	\begin{abstract}
		Motivated by models for thin films coating cylinders in two physical cases proposed in  \cite{K} and \cite{FK}, we analyze the dynamics of corresponding thin film models.  The models are governed by nonlinear, fourth-order, degenerate, parabolic PDEs.  We prove, given positive and suitably regular initial data, the existence of weak solutions in all length scales of the cylinder, where all solutions are only local in time.  We also prove that given a length constraint on the cylinder, long-time and global in time weak solutions exist.  This analytical result is motivated by numerical work on related models of Reed Ogrosky \cite{ogrosky2013modeling} in conjunction with the works \cite{camassa2012ring,COO,camassa2017viscous,CMOV}.
	\end{abstract}
		\maketitle
	
	\section{Introduction}
	The analysis of liquid films is an area of mathematical research that has many applications, ranging from biological systems to engineering and has been a rich area of research over the last three decades. Generically, the films have one free boundary whose evolution is determined by the relationship between external forces and the surface tension of the free surface itself.  Many modeling and numerical studies have been done in order to understand these flows in different parameters and geometrical setups.  In particular \cite{Be} and \cite{Li} study films along an inclined plane and \cite{F},\cite{LL},\cite{COO} consider films in the exterior or interior of vertically oriented tubes.  The most significant physical difference between these two geometries is the free surface's azimuthal curvature dictating the surface stress in the cylindrical setting. The interior case of the cylindrical geometry is studied extensively in \cite{CMOV}. A specific class of the films, called thin films,  exploit the ratio between the thicknesses of the film and the cylinder.  In \cite{F}, an evolution equation is derived for a thin film coating either the outside or the inside of a cylinder.  This model was further studied in \cite{FK} and is explained in greater detail below.  Thin films equations have also been studied in the frameworks of the generalized Kuramoto-Sivashinsky equation and the Cahn-Hilliard equation \cite{BJNRZ},\cite{HR}.
	
	Much work in the area has drawn on the machinery developed in \cite{BF}, where the equation
	\begin{equation}\label{BernisFriedman}
		h_t + (f(h)h_{xxx})_x = 0, \quad f(h) = f_0(h)|h|^n, \quad 0 < f_0 \in C^{1+\alpha}(\R^1) , \alpha \in (0,1), \text{ and } n \geq 1
	\end{equation}
	is examined.  In particular, an energy-entropy method is developed in \cite{BF} in order to prove the existence of weak solutions to \eqref{BernisFriedman}, where Neumann boundary conditions on a finite interval are assumed.  Many others have used these tools in order to make progress from the analytical standpoint, including well-posedness, existence of weak solutions, finite-time blow-up, finite speed of propagation, and waiting time phenomena for various thin film type equations.  In \cite{G} an equation modeling the flow of a thin film over an inclined plane is analyzed and global in time existence of weak and strong solutions is given. There has also been a fine collection of work in proving finite-time blow-up in some of the models.  In \cite{CT}, an equation modeling the spreading of a thin film over a flat solid surface and studied and a blow-up result is proved.  Finite-time blow-up can also be seen  in \cite{BP}. A comprehensive discussion of the relationship between scaling properties and singularity formation can be found in \cite{GKMS}.
	
	Though there has been some work done in more general settings \cite{SW}, thin films coating a cylinder have been studied extensively.  Eres, Schwartz, and Weidner provided models and numerical work for a stationary, horizontally oriented cylinder in the presence of gravity \cite{ESW}.  Aside from modeling and numerical work, much analytic progress has been made by Bertozzi, Chugunova, Pugh, and Taranets.  A well known result is \cite{BP98}, in which a pure thin film equation is studied. Dynamics of thin films with convection on the exterior surface of a horizontally oriented cylinder rotating about it's axis of symmetry have also been studied and there are arguments for long-time existence of weak solutions \cite{CT2},\cite{CPT},\cite{Ta}.  A key difference is that the regularization in \cite{BP98} features global in time existence, while this is not generically guaranteed when a convection term is present.  In particular, any lower order terms complicate the derivation of corresponding a priori estimates and qualitative solution behavior. 
	
	In this paper, we study the dynamics of an incompressible thin fluid film on the exterior of a cylinder.  In particular we consider two specific one dimensional models.  The first model (Model I), derived in \cite{K}, is given by the initial boundary value problem
	\begin{equation}\label{Kerchman}
		\begin{cases}
			&h_t = - hh_x - S\left[ h^3(h_x + h_{xxx}) \right]_x \t{ in $Q_T$},\\
			&h(x,0) = h_0(x) \in H^1(\Omega), \\
			& \partial_x^j h(-a,\cdot) = \partial_x^j h(a,\cdot) \t{ for $t \in (0,T)$}, \ j = 0,1,2,3, 
		\end{cases}
	\end{equation}
	where $\Omega = (-a,a)$ is a bounded interval in $\R$ and $Q_T = \Omega \times (0,T)$.  The equation models the situation in which the cylinder is horizontally oriented, a horizontally directed air flow is present without gravity, and the cylinder is fully coated so that the only free boundary is that where the surface of the fluid meets the air.  Here, $h$ is the thickness of the film with initial value $h_0$ and $x$ is the longitudinal position.
%	A schematic diagram (originally found in \cite{K}) is provided below in Figure \ref{fig:M1Cap}. 
%	\begin{center}
%	 	\includegraphics[scale = 0.23]{KerchmanModel.png}
%	 	\captionof{figure}{Model I (Image from \cite{K})}\label{fig:M1Cap}
%	\end{center}
	Model II, derived in \cite{FK}, is given by the initial boundary value problem
	\begin{equation}\label{Frenkel}
	\begin{cases}
	&h_t =  -2h^2h_x - S \left[ h^3(h_x + h_{xxx}) \right]_x \t{ in $Q_T$}, \\
	&h(x,0) = h_0(x) \in H^1(\Omega), \\
	& \partial_x^j h(-a,\cdot) = \partial_x^j h(a,\cdot) \text{ for $t \in (0,T), \ j = 0,1,2,3$}.
	\end{cases}
	\end{equation}
	This equation models the thickness of a thin film fully coating a vertically oriented cylinder in the presence of gravity.   
%	A  schematic diagram (originally found in \cite{FK}) is provided below in Figure \ref{fig:M2Cap}.
%	\begin{center}
%		\includegraphics[scale = 0.45]{FrenkelModel.png}
%		\captionof{figure}{Model II (Image from \cite{FK})}\label{fig:M2Cap}
%	\end{center}
	In each model, $S$ is a modified Weber number. The rightmost terms in each equation represent the effects of surface tension in the azimuthal and axial directions, respectively, and the first terms on the right hand sides represent the forces acting on the films, e.g., air flow in Model I and gravity in Model II. Schematic diagrams of each model and respective coordinates can be found in \cite{K} and \cite{FK}.

	The main result of this work is to establish the following theorem.
	\begin{Thm}
	\label{thm:main}
		For sufficiently regular initial data $h_0 \geqslant 0$, there exist local in time weak solutions to \eqref{Kerchman} and \eqref{Frenkel}.  In addition, given $\Omega \subset \left( -\frac{\pi}{2}, \frac{\pi}{2} \right),$ there exist long-time non-negative solutions to \eqref{Kerchman} and global in time non-negative solutions to \eqref{Frenkel}.
	\end{Thm}

The goal of the remainder to the result is to provide full details on the function spaces, estimates, and notion of weak solutions required for analysis of both Model I and Model II in order to establish Theorem \ref{thm:main} precisely.  Furthermore, in all cases, we prove the non-negativity property, i.e. positive initial conditions yields non-negative solutions.   In Sections \ref{M1Preliminaries} and \ref{M1EnergyEstimates}, we introduce a properly regularized version of \eqref{Kerchman} and analyze it using energy estimates.  When $\Omega \subset \left( -\frac{\pi}{2}, \frac{\pi}{2} \right),$ we prove the existence of long-time weak solutions to Model I.   Then, in Section \ref{M1WeakSolutions}, we define carefully weak solutions and demonstrate that the limit of the solution to the regularized problem exists in the proper fashion.   We prove in Section \ref{M1Non-Negativity} that the limit is indeed non-negative.  Then, in Section \ref{M2}, we examine Model II and give a brief description of how to prove the local in time existence of weak solutions by proving analogous energy estimates.   We then provide a proof for the existence of global in time solutions to Model II. 

Though many of the components of the proof are naturally analogous to those for Model I, the energy estimates are treated with a modified approach and are given in details.  We note that though one may be inclined to consider using an interpolation inequality as in \cite{BP98} in order to prove long-time and global estimates, these inequalities are less quantitative than the Poincar\'{e} inequality and therefore cannot give sharp results. Finally, in section \ref{FutureWork} we discuss future work in this analysis, including natural extensions of the arguments found here to long-wave models and a mixing of Model I and Model II.

%%%%%%%%%%%%%%%%%%%%%%%%%%%%%%%%%%%%%%%%%%%%%%%%%%%%%%%%%%%%%%%%%%%%%%%%%%%%%%%%%%%%%%%%%%%%
%%%%%%%%%%%%%%%%%%%%%%%%%%%%%%%%%%%%%%%%%%%%%%%%%%%%%%%%%%%%%%%%%%%%%%%%%%%%%%%%%%%%%%%%%%%%
	
	\section{Model I Preliminaries}\label{M1Preliminaries}
	
	One notices that \eqref{Kerchman} is degenerate if $h$ vanishes at any point in the domain, and in order for the equation to be uniformly parabolic, it must be the case that $h \geqslant \delta$ in $Q_T$ for some $\delta > 0$.  In order to remedy this, one may consider the regularized problem
	\begin{equation}\label{KerchmanPert}
	\begin{cases}
	&h_{\ve,t} = -h_\ve h_{\ve,x} - S\left[ \left( |h_\ve|^3 + \ve \right)(h_{\ve,x} + h_{\ve,xxx}) \right]_x \t{ in $Q_T$}, \\
	&h_\ve(x,0) = h_{0,\ve}(x) \in C^{4 + \gamma}(\Omega) \t{ for some } \gamma \in (0,1), \\
	& \partial_x^j h_\ve(-a,\cdot) = \partial_x^j h_\ve(a,\cdot) \t{ for $t \in (0,T)$}, \ j = 0,1,2,3, 
	\end{cases}
	\end{equation}
	Observe that the right hand sides of both \eqref{Kerchman} and \eqref{KerchmanPert} have a gradient form, i.e. $h_t = - \t{div} J$.  This fact and the periodic boundary conditions tell us that integrating over $Q_T$ yields
	\begin{equation*}
	\int_\Omega h_\ve(x,T) \ dx = \int_\Omega h_{0,\ve} \ dx =: M_\ve < \infty
	\end{equation*}
	for each $0 \leq T \leq T_\ve$ and for each $\ve > 0$.  In other words, \eqref{Kerchman} and \eqref{KerchmanPert} are both conservation laws and conserve $\int_\Omega h(x,t) \ dx$ over time.  We assume that $h_{0,\ve} \to h_0$ strongly in $H^1(\Omega)$.  Then we can bound $M_\ve$ uniformly by $M := \int_\Omega h_0 \ dx > 0$.  Thus for $\ve > 0$ sufficiently small we have $0 < \int_\Omega h_{0,\ve} \ dx \leq M < \infty$.
	
	\subsection{Functionals}\label{Main Functionals}
	Here we define some different energy terms:
	\begin{equation*}
		E_0(h) = \frac{1}{2} \int_\Omega h^2 \ dx, \quad	E_1(h) = \frac{1}{2} \int_\Omega h_x^2 \ dx, \quad \mc{E}(h) = \frac{1}{2} \int_{\Omega} (h_x^2 - h^2) \ dx.
	\end{equation*}
	We also define the functions $g_\ve$ and $G_\ve$ by
	\begin{align}
		\label{g-eps} g_\ve(s) &= -\int_s^A \frac{dr}{|r|^3 + \ve}, \\
		\label{G-eps} G_\ve(s) &= - \int_s^A g_\ve(r) \ dr,
	\end{align}
	where $A>0$ is a finite real number to be specified later.
	
	The use of these functionals naturally draws on their use in \cite{BF}.  There are some useful  statements \cite{BF} makes of $g_\ve$ and $G_\ve$:
	\begin{alignat*}{2}
		&G_\ve'(s) = g_\ve(s), &&G_\ve''(s) = g_\ve'(s) = \frac{1}{|s|^3 + \ve}, \\
		&g_\ve(s) \leq 0 \ \forall s \leq A, &&G_\ve(s) \geq 0 \ \forall s \in \R, \\
		&G_\ve(s) \leq G_0(s) \ \forall s \in \R,
	\end{alignat*}
	where $G_0 = \lim_{\ve \to 0} G_\ve.$  Finally,
	\begin{equation*}
		G_0(s) = \frac{1}{2s} - \frac{1}{A} + \frac{s}{2A^2} = \frac{(A-s)^2}{2A^2s} \geq 0 \ \forall s > 0 .
	\end{equation*}
	
	\subsection{Model I Energy Identities}
	We first work with the regularized equation \eqref{KerchmanPert} to derive a priori estimates.  To begin, we draw on general parabolic theory in order to demonstrate that the perturbed equation is well-posed.  Consider the operator 
	\begin{equation*}
		P_\ve(x,y,t) = -y\dx_x - S\dx_x \left[ (|y|^3 + \ve)(\dx_x + \dx_{xxx})\right].
	\end{equation*}
	Then the equation 
	\begin{equation*}
		h_t = P_\ve h
	\end{equation*}
	is uniformly parabolic in a region $Q = [0,R] \times \Omega \times [0,T_\ve]$ in the sense of Petrovsky \cite{E} as the characteristic equation
	\begin{equation*}
		D(\lambda,\sigma)= -S(|y|^3 + \ve)\sigma^4 - \lambda = 0
	\end{equation*}
	has root $\lambda = -S(|y|^3 + \ve)\sigma^4$ which can be bounded above by $-\delta(\ve) < 0$ so long as $|y| < R := R(\ve)$ and $S > 0$.  Theorem 7.3 in \cite{E} tells us that there exists a unique classical solution $h_\ve \in C_{x,t}^{4 + \gamma,1 + \frac{\gamma}{4}}(Q_{T_\ve})$ to \eqref{KerchmanPert}, where $\gamma \in (0,1)$. In the rest of this section and the following two sections we will write $h = h_\ve$.

	Multiplying \eqref{KerchmanPert} by $h$ and integrating over $\Omega$, we obtain
	\begin{align*}
		\int_\Omega hh_t \ dx &= -\frac{1}{2} \int_\Omega h^2h_x \ dx - S\int_\Omega \left[ \left( |h|^3 + \ve \right) (h_x + h_{xxx})\right]_x h \ dx \\
		&= S\int_\Omega \left(|h|^3 + \ve \right)h_x^2 \ dx + S\int_\Omega \left( |h|^3 + \ve \right) h_xh_{xxx} \ dx,
	\end{align*}
	where the last line uses integration by parts and periodic boundary conditions.
	Similarly, we can multiply \eqref{KerchmanPert} by $-h_{xx}$ and integrate over $\Omega$ to see that
	\begin{align*}
	\int_\Omega h_xh_{x,t} \ dx &= -\int_\Omega h_th_{xx} \ dx \\
	&= \int_\Omega hh_xh_{xx} \ dx + S\int_\Omega \left[ \left( |h|^3 + \ve \right) \left( h_x + h_{xxx} \right) \right]_x h_{xx} \ dx \\
	&= \frac{1}{2} \int_\Omega \left( h^2 \right)_x h_{xx} \ dx - S\int_\Omega \left[ \left( |h|^3 + \ve \right) \left( h_x + h_{xxx} \right) \right] h_{xxx} \ dx  \\
	&= -\frac{1}{2} \int_\Omega h^2 h_{xxx} \ dx - S\int_\Omega \left( |h|^3 + \ve \right)h_xh_{xxx} \ dx - S\int_\Omega \left( |h|^3 + \ve \right) h_{xxx}^2 \ dx.
	\end{align*}
	Adding the left hand and right hand sides of the chains of equalities, we have
	\begin{equation}\label{V'}
		\frac{1}{2} \frac{d}{dt} ||h||_{H^1(\Omega)}^2 + S\int_\Omega \left( |h|^3 + \ve  \right)h_{xxx}^2 \ dx = -\frac{1}{2} \int_\Omega h^2h_{xxx} \ dx + S\int_\Omega \left( |h|^3 + \ve \right)h_x^2.
	\end{equation}

%%%%%%%%%%%%%%%%%%%%%%%%%%%%%%%%%%%%%%%%%%%%%%%%%%%%%%%%%%%%%%%%%%%%%%%%%%%%%%%%%%%%%%%%%%%%
%%%%%%%%%%%%%%%%%%%%%%%%%%%%%%%%%%%%%%%%%%%%%%%%%%%%%%%%%%%%%%%%%%%%%%%%%%%%%%%%%%%%%%%%%%%%
	
	\section{Model I Energy Estimates}\label{M1EnergyEstimates}
	
	\subsection{Local in Time Estimates}\label{M1ShortTimeEstimates}
	We can obtain uniform bounds on $||h_\ve(\cdot,T)||_{H^1(\Omega)}^2$ for $\ve > 0$ and $T > 0$ sufficiently small.
	\begin{Lem}\label{KSTEstLem}
		Suppose $h_0$ as in \eqref{Kerchman} and let $h_{0,\ve} \to h_0$ strongly in $H^1(\Omega)$.  Let $h_\ve$ be a solution to \eqref{KerchmanPert} in $Q_{T_\ve}$.  Then there is a time $T_\t{loc} > 0$ such that $h_\ve$ satisfies a priori estimate
		\begin{equation*}
			||h_\ve(\cdot,T)||_{H^1(\Omega)}^2 \leq 2^{2/3} \max \left\{ 1, ||h_0||_{H^1(\Omega)}^2 \right\}
		\end{equation*}
		for $\ve > 0$ and $0 \leq T \leq T_\t{loc}.$
		\begin{proof}
			Let $\ve > 0$ and let $h := h_\ve$ be a solution to \eqref{KerchmanPert}.  Recalling \eqref{V'}, we can bound using Cauchy's inequality and the compact embedding of $H^1$ in $L^\infty$:
			\begin{align*}
				& \frac{1}{2} \frac{d}{dt} ||h||_{H^1(\Omega)}^2 + S\int_\Omega \left( |h|^3 + \ve  \right)h_{xxx}^2 \ dx \\
				& \hspace{1cm} \leq \frac{S}{4} \int_\Omega \left( |h|^3 + \ve \right) h_{xxx}^2 \ dx + \frac{1}{4S} \int_\Omega |h| \ dx + S\int_\Omega \left( |h|^3 + \ve \right) h_x^2 \ dx.
			\end{align*}
			As 
			\begin{equation*}
				\int_\Omega |h| \ dx \leq |\Omega|^{1/2} \left( \int_\Omega h^2 \ dx \right)^{1/2},
			\end{equation*}
			then 
			\begin{equation*}
					\frac{1}{2} \frac{d}{dt} ||h||_{H^1(\Omega)}^2 + 
					\frac{3S}{4}\int_\Omega \left( |h|^3 + \ve  \right)h_{xxx}^2 \ dx \leq C_S||h||_{H^1(\Omega)}^5 + S\ve||h||_{H^1(\Omega)}^2 + \frac{|\Omega|^{1/2}}{4S} ||h||_{H^1(\Omega)}.
			\end{equation*}
			Setting $V_\ve(t) := \max \left( 1, ||h_\ve(\cdot,t)||_{H^1(\Omega)}^2 \right)$, it follows that $V_\ve$ satisfies
			\begin{equation*}
				V_\ve'(t) \leq C_SV_\ve^{5/2}(t).
			\end{equation*}
			Dividing by $V_\ve^{5/2}(t)$ and integrating yields
			\begin{equation}\label{STBound}
				V_\ve(t) \leq \left( V_\ve^{-3/2}(0) - \frac{3}{2}C_St \right)^{-2/3}
			\end{equation}
			for $0 \leq t < T^* = \frac{2}{3C_S}V_\ve^{-3/2}(0).$  As $h_{0,\ve} \to h_0$ strongly in $H^1(\Omega)$, \eqref{STBound} implies that $||h_\ve||_{H^1(\Omega)}$ is uniformly bounded for $0 \leq T \leq T_\t{loc} = \frac{1}{3C_S}V^{-3/2}(0)$ and independent of $\ve > 0$.
		\end{proof}
	\end{Lem}
		
	\subsection{Long-time Estimates on $\Omega \subset \left( -\frac{\pi}{2},\frac{\pi}{2} \right)$}\label{M1LongTimeEstimates}
	
	Here, we require that for $\Omega = (-a,a)$, we have $a < \frac{\pi}{2}$.  We assume that $a < \frac{\pi}{2}$ in order to properly bootstrap $||h||_{L^2(\Omega)}$ by $||h_x||_{L^2(\Omega)}$ in Lemma \ref{PertGronwallBound}, which is done using the quantitative nature of the Poincar\'{e} inequality in $L^2(\Omega)$ with the exact constant. Note that a scaling does not eliminate the condition on smallness of domains. We consider \eqref{KerchmanPert} and fix $\ve > 0$.  The existence theory in \cite{E} (Theorem 6.3 on page 302) tells us that there is a classical solution $h_\ve \in C^{4+ \gamma,1 + \gamma/4}(Q_{\tau_\ve})$ to \eqref{KerchmanPert} for some small time $\tau_\ve > 0$.  It is further demonstrated in \cite{E} (Theorem 9.3 on page 316) that if we have a priori control $||h_\ve||_{L^\infty(Q_{T_\ve})} \leq A$ and control on the H\"{o}lder norms in $C_{x,t}^{1/2,1/8}(Q_{T_\ve})$ for some $T > \tau_\ve$, then, in fact, $h_\ve$ can be continued in time as a classical solution to \eqref{KerchmanPert} on $Q_T$.  We use the the functional $\mc{E}(h) = \frac{1}{2} \int_{\Omega} (h_x^2 - h^2) \ dx$ to demonstrate such control.
	
	Before proceeding we require the Gr\"{o}nwall type inequality found, for example, in \cite{Gy}:
	\begin{Lem}\label{NLGronwall}
		Suppose that $y(t)$ satisfies the inequality
		\begin{equation*}
			y(t) \leq at + b + c \int_{t_0}^t \ g(y(s)) \ ds \ \forall t \geq t_0 \geq 0,
		\end{equation*}
		where $y$ is a non-negative continuous function, $g$ is a positive nondecreasing function, and $a,b,c > 0$.  Then 
		\begin{equation*}
			y(t) \leq G^{-1} \left\{ G(at_0 + b) + \left( \frac{a}{g(at_0 + b)} + c \right)(t-t_0) \right\},
		\end{equation*}
		where 
		\begin{equation*}
			G(t) = \int_\eta^t \frac{ds}{g(s)} \ \text{for } \eta, t > 0.
		\end{equation*}
			\end{Lem}

		\begin{proof}
			Begin by defining
			\begin{equation*}
				w(t) = at + b + c \int_{t_0}^t \ g(y(s)) \ ds.
			\end{equation*}
			Because $g$ is nondecreasing, $g(y(t)) \leq g(w(t))$.  Notice that
			\begin{equation*}
				w'(t) = a + cg(y(t)).
			\end{equation*}
			whence it follows that
			\begin{equation*}
				w'(t) \leq a + cg(w(t)).
			\end{equation*}
			Using the fact that $g > 0,$ we obtain
			\begin{equation*}
				\frac{w'(t)}{g(w(t))} \leq \frac{a}{g(w(t))} + c.
			\end{equation*}
			Again using that $g > 0$ is nondecreasing and noticing that $w$ is non-decreasing, it follows that 
			\begin{equation*}
				\frac{w'(t)}{g(w(t))} \leq \frac{a}{g(at_0 + b)} + c.
			\end{equation*}
			Integrating yields
			\begin{equation*}
				G(w(t)) \leq G(w(t_0)) + \left( \frac{a}{g(at_0 + b)} + c \right)(t- t_0),
			\end{equation*}
			and applying the inverse of $G$ yields the result so long as $t \in [t_0,T^*]$ where $T$ is chosen such that 
			\[ G(w(t_0)) + \left( \frac{a}{g(at_0 + b)} + c \right)(T-t_0) \in \t{Dom}(G^{-1}) \ \forall T \in [t_0,T^*].\]
		\end{proof}
	
	\begin{Lem}\label{PertGronwallBound}
		Fix $\ve > 0$ and let $h_\ve$ be a solution of \eqref{KerchmanPert} up to time  $T > 0$.   Then $h_\ve$ satisfies the a priori estimate 
		\begin{equation}\label{LongtimeUnifBound1}
			||h_{\ve,x} (\cdot,t)||_{L^2(\Omega)}^2 \ dx \leq K_0 + K_1t + K_2 t^2.
		\end{equation}
			\end{Lem}

		\begin{proof}
			Multiplying \eqref{KerchmanPert} by $-h_\ve - h_{\ve,xx}$, integrating over $Q_T$, integrating by parts, and using the periodic boundary conditions, one obtains
			\begin{equation}\label{EnergyEst1}
				\mc{E}(h_\ve(\cdot,T)) + S\iint_{Q_T} \left( |h_\ve|^3 + \ve \right) \left( h_{\ve,x} + h_{\ve,xxx} \right)^2 \ dx \, dt = \vE(h_{0,\ve}) - \frac{1}{2} \iint_{Q_T} h_\ve^2  h_{\ve,xxx} \ dx \, dt.
			\end{equation}
			This implies
			\begin{align*}
			& 	||h_x(\cdot,T)||_{L^2(\Omega)}^2 + 2S\iint_{Q_T} \left( |h|^3 + \ve \right) \left( h_x  + h_{xxx} \right)^2 \ dx \ dt \\
			& \hspace{1cm} = ||h(\cdot,T)||_{L^2(\Omega)}^2 +  2\mc{E}(h_{0,\ve}) - \iint_{Q_T} h^2 h_{xxx} \ dx \ dt.
			\end{align*}
			Applying the Poincar\'{e} inequality to $h(x,T)$, we obtain
			\begin{align}\label{LTEst1}
				& \left( 1 - \left( \frac{|\Omega|}{\pi} \right)^2 \right) ||h_x(\cdot,T)||_{L^2(\Omega)}^2 + 2S\iint_{Q_T} \left( |h|^3 + \ve \right) \left( h_x  + h_{xxx} \right)^2 \ dx \ dt \\
				& \hspace{1cm}  \leq  2\mc{E}(h_{0,\ve}) + \frac{M_\ve^2}{|\Omega|} - \iint_{Q_T} h^2 h_{xxx} \ dx \ dt. \notag
			\end{align}
			Now, observe that we can bound the integral on the right hand side:
			\begin{align*}
			&	-\iint_{Q_T} h^2h_{xxx} \ dx \ dt \overset{\mathrm{\begin{array}{c} \t{periodic} \\ \t{boundary} \\ \t{conditions} \end{array}}}{=} -\iint_{Q_T} h^2 \left( h_x + h_{xxx} \right) \\
				& \hspace{.1cm} \overset{\mathrm{\begin{array}{c} \t{Cauchy's} \\ \t{inequality} \end{array}}}{\leq} \frac{S}{2} \iint_{Q_T} \left( |h|^3 + \ve \right) \left( h_x + h_{xxx} \right)^2 \ dx \ dt + \frac{1}{2S} \iint_{Q_T} |h| \ dx \ dt.
			\end{align*}
			Therefore, it follows from \eqref{LTEst1} that
			\begin{align}\label{LTEst2}
				 & \left( 1 - \left( \frac{|\Omega|}{\pi} \right)^2 \right) ||h_x(\cdot,T)||_{L^2(\Omega)}^2 + \frac{3S}{2}\iint_{Q_T} \left( |h|^3 + \ve \right) \left( h_x  + h_{xxx} \right)^2 \ dx \ dt \\
				& \hspace{1cm} \leq  2\mc{E}_\ve(0) + \frac{M_\ve^2}{|\Omega|} + \frac{1}{2S} \iint_{Q_T} |h| \ dx \ dt. \notag
			\end{align}
			Again, we can use the Cauchy-Schwarz and Poincar\'{e} inequalities to see that
			\begin{align*}
				\iint_{Q_T} |h| \ dx \ dt &\leq |\Omega|^{1/2} \int_0^T \left( \int_\Omega h^2 \ dx  \right)^{1/2} \ dt \\
				&\leq |\Omega|^{1/2} \int_0^T \left( \left( \frac{|\Omega|}{\pi} \right)^2 \int_\Omega h_x^2 \ dx + \frac{M_\ve^2}{|\Omega|}  \right)^{1/2} \ dt \\
				&\leq \frac{|\Omega|^{3/2}}{\pi} \int_0^T ||h_x(\cdot,t)||_{L^2(\Omega)} \ dt +  M_\ve T.
			\end{align*}
			Applying this bound to \eqref{LTEst2} we have
			\begin{equation*}
				||h_x(\cdot,T)||_{L^2(\Omega)}^2 \leq \alpha(T) + K\int_0^T ||h_x(\cdot,t)||_{L^2(\Omega)} \ dt,
			\end{equation*}
			where 
			\begin{equation*}
				\begin{split}
					\alpha(T) &= \left( 1 - \left( \frac{|\Omega|}{\pi} \right)^2  \right)^{-1}  \left( 2\mc{E}(h_{0,\ve}) + \frac{M_\ve^2}{|\Omega|}  + \frac{M_\ve}{2S} T  \right),  \\
					K &= \left( 1 - \left( \frac{|\Omega|}{\pi} \right)^2  \right)^{-1} \frac{|\Omega|^{3/2}}{2S\pi}.
				\end{split} 
			\end{equation*}	
			An application of Lemma \ref{NLGronwall} completes the proof of \eqref{LongtimeUnifBound1}, with 
			\begin{equation*}
					K_0= \alpha(0), \ \
					K_1 = \alpha'(0) + K\sqrt{\alpha(0)}, \ \
					K_2 = \frac{1}{4}\left( \frac{\alpha'(0)}{\sqrt{\alpha(0)}} + K \right)^2.
			\end{equation*}
%			To prove \eqref{LongtimeUnifBound2}, start with \eqref{EnergyEst1} and apply Young's inequality to obtain
%			\begin{equation}
%				\begin{split}
%				\mc{E}_(h(\cdot,T)) + \frac{1}{2} \iint_{Q_T} \left( |h|^3 + \ve \right)(h_x + h_{xxx})^2 \ dx dt &\leq \mc{E}(h_{0,\ve}) + \iint_{Q_T} |h|^3 \ dx \ dt \\
%				&\leq \mc{E}(h_{0,\ve}) + \int_0^T ||h(\cdot,t)||_{L^\infty(\Omega)}^3 \ dt.
%				\end{split}
%			\end{equation}
%			Using the Sobolev and Poincar\'{e} inequalities followed by \eqref{LongtimeUnifBound1} yields the result.
		\end{proof}
	
	Application of Poincar\'{e} and Sobolev inequalities immediately implies that for any finite time $T$, we have a priori bound for $||h_\ve||_{L^\infty(Q_T)}$.
	
	\subsection{H\"{o}lder Continuity of $\{ h_\ve \}_{\ve > 0}$}\label{M1Holder}
	
	Let $T < \infty$ be a uniform time of existence for a family of solutions $\{ h_\ve \}_{\ve > 0}$.  Using the uniform boundedness of $||h_\ve||_{H^1(Q_T)}$, an application of Morrey's inequality (\cite{Ev} page 282) implies that $h_{\ve}(\cdot,t)$ are uniformly bounded in  $C^{1/2}(\bar{\Omega})$ for $0 \leq \ve \leq \ve_0$, $0 \leq t \leq T$, i.e. there is a constant $K_3$ such that
	\begin{equation}\label{xHolder}
		|h_\ve(x_1,t) - h_\ve(x_2,t)| \leq K_3|x_1 - x_2|^{1/2},
	\end{equation}
	where the constant $K_3$ is independent of $\ve$.
	
	\begin{Lem}\label{KtHolderLem}
		There is a constant $M < \infty$ so that for every $0 \leq \ve \leq \ve_0$ and $0 \leq t_1 < t_2 \leq T$, $h_\ve$ satisfies
		\begin{equation}\label{KtHolder}
			|h_\ve(x_0,t_1) - h_\ve(x_0,t_2)| \leq M |t_1 - t_2|^{1/8}
		\end{equation}
		for each $x_0 \in \Omega.$
		\begin{proof}
			Suppose that
			\begin{equation}\label{tLower}
				|h_\ve(x_0,t_1) - h_\ve(x_0,t_2)| > M |t_2 - t_1|^{1/8}
			\end{equation}
			for some $x_0 \in \Omega$ and some $0 \leq t_1 < t_2 \leq T$.  We will derive an upper bound for $M$ independent of $\ve$.  Without loss of generality, assume that $h_\ve(x_0,t_2) > h_\ve(x_0,t_1)$.
			
			Following the work in \cite{BF}, we define $\xi_0 \in C_0^\infty$ so that $\xi_0$ is even, $\xi_0(x) = 1$ if $0 \leq x \leq \frac{1}{2}$, $\xi_0(x) = 0$ if $x \geq 1$, and $\xi_0'(x) \leq 0$ for $x \geq 0$.  Setting
			\begin{equation*}
				\xi(x) = \xi_0 \left( \frac{x - x_0}{(M^2/16K_3^2)(t_2 - t_1)^{2\beta}}  \right),
			\end{equation*}
			where $\beta = \frac{1}{8}$.  It follows that
			\begin{equation}\label{xi}
				\xi(x) = 
					\begin{cases}
						0 &\t{ if } |x - x_0| \geq \frac{M^2}{16K_3^2}(t_2 - t_1)^{2\beta} \\
						1 &\t{ if } |x - x_0| \leq \frac{1}{2}\frac{M^2}{16K_3^2}(t_2 - t_1)^{2\beta}.
					\end{cases}
			\end{equation}
			
			We next define $\theta_\delta$ by
			\begin{equation*}
				\theta_\delta(t) = \int_{-\infty}^t \theta_\delta'(s) \ ds,
			\end{equation*}
			where $\theta_\delta'$ is given by
			\begin{equation*}
				\theta_\delta'(t) =
				\begin{cases}
					\frac{1}{2\delta} &\t{ if } |t- t_2| < \delta \\
					-\frac{1}{2\delta} &\t{ if } |t - t_1|< \delta \\
					0 &\t{otherwise}
				\end{cases}
			\end{equation*}
			and $0 < \delta < \frac{t_2 - t_1}{2}.$  It is easy to see that $\theta_\delta$ is Lipschitz continuous and that $|\theta_\delta| \leq 1$.  Furthermore, $\theta_\delta = 0$ near $t = 0$ and $t = T$ provided $\delta$ is small enough.
			
			Setting $\phi(x,t) = \xi(x)\theta_\delta(t),$ it is clear that integration by parts yields
			\begin{equation*}
				\iint_{Q_T} h_\ve \phi_t \ dx \, dt = - \iint_{Q_T} f_\ve \phi_x \ dx \, dt,
			\end{equation*}
			where $f_\ve = \frac{h_\ve^2}{2} + S\left( |h|^3 + \ve \right)\left( h_{\ve,x} + h_{\ve,xxx} \right)$.  Using the definition of $\phi$, we see that
			\begin{equation}\label{TestParts}
				\iint_{Q_T} h_\ve(x,t) \xi(x) \theta_\delta'(t) \ dx \, dt = - \iint_{Q_T} f_\ve(x,t)\xi'(x) \theta_\delta(t) \ dx \, dt.
			\end{equation}
			We first work with the left hand side of \eqref{TestParts}.  Taking the limit as $\delta$ tends to $0$, it is clear that
			\begin{equation}\label{deltalimit}
				\lim_{\delta \to 0} \iint_{Q_T} h_\ve(x,t) \xi(x) \theta_\delta'(t) \ dx \, dt = \int_\Omega \xi(x) \left( h_\ve(x,t_2) - h_\ve(x,t_1) \right) \ dx.
			\end{equation}
			We will estimate \eqref{deltalimit} from below.  Because of \eqref{xi}, it is clear that we must only consider values $x$ such that 
			\begin{equation}\label{xRestrict}
				|x - x_0| \leq \frac{M^2}{16K_3^2}(t_2 - t_1)^{2\beta}. 
			\end{equation}
			Note that for such values of $x$, we have
			\begin{align*}
				& h_\ve(x,t_2) - h_\ve(x,t_1) \\
				& \hspace{.5cm} = \left( h_\ve(x,t_2) - h_\ve(x_0,t_2) \right) + \left( h_\ve(x_0,t_2) - h_\ve(x_0,t_1) \right) + \left( h_\ve(x_0,t_1) - h_\ve(x,t_1) \right) \\
				& \hspace{1cm} \geq -2K_3|x - x_0|^{1/2} + M(t_2 - t_1)^\beta \t{, by \eqref{xHolder} and \eqref{tLower}.} \\
				& \hspace{1.5cm}  \geq \frac{M}{2}(t_2 - t_1)^\beta \t{, by \eqref{xRestrict}}.
			\end{align*}
			If we assume that $\left\{ \xi = 1 \right\} \subset \Omega$, then we have
			\begin{equation*}
				\int_{\Omega} \xi(x) \left( h_\ve(x,t_2) - h_\ve(x,t_2) \right) \ dx \geq \frac{M}{2}(t_2 - t_1)^\beta \frac{M^2}{16K_3^2}(t_2 - t_1)^{2\beta}.
			\end{equation*}
			
			We now work to bound the right hand side of \eqref{TestParts}.  Observe that by Cauchy-Schwarz
			\begin{align*}
				& \left| \iint_{Q_T} f_\ve(x,t) \xi'(x) \theta_\delta(t) \ dx \, dt \right| \leq ||f_\ve||_{L^2(Q_T)} \left( \int_\Omega \xi'(x)^2 \ dx \int_0^T \theta_\delta(t)^2 \ dt  \right)^{1/2} \\\
				& \hspace{.25cm} = ||f_\ve||_{L^2(Q_T)} \left( \int_\Omega \left( \frac{d}{dx} \xi_0 \left( \frac{x-x_0}{(M^2/16K_3^2)(t_2 - t_1)^{2\beta}} \right) \right)^2 \ dx \int_0^T \theta_\delta(t)^2 \ dt \right)^{1/2} \\
				&\hspace{.25cm} = \frac{1}{(M^2/16K_3^2)(t_2 - t_1)^{2\beta}}||f_\ve||_{L^2(Q_T)} \\
				&\hspace{.5cm}  \times \left( \int_\Omega \xi_0' \left( \frac{x-x_0}{(M^2/16K_3^2)(t_2 - t_1)^{2\beta}} \right)^2 \ dx \int_0^T \theta_\delta(t)^2 \ dt \right)^{1/2} \\
				&\hspace{.25cm} = \frac{1}{(M^2/16K_3^2)(t_2 - t_1)^{2\beta}}||f_\ve||_{L^2(Q_{T_\t{loc}})} \\
				& \hspace{.5cm} \times  \left( \int_\Omega \xi_0' \left( \frac{x-x_0}{(M^2/16K_3^2)(t_2 - t_1)^{2\beta}} \right)^2 (t_2 - t_1 + 2\delta) \right)^{1/2} \\
				&\hspace{.25cm}  \leq \frac{1}{(M^2/16K_3^2)(t_2 - t_1)^{2\beta}}||f_\ve||_{L^2(Q_{T_\t{loc}})} \frac{C\sqrt{2}M}{4K_3} (t_2 - t_1 + 2\delta)^{1/2},
			\end{align*}
			where we obtain the last inequality from the support of $\xi'(x)$ and taking 
			\[ C = \sup_{x \in \Omega} \xi_0' \left( \frac{x-x_0}{(M^2/16K_3^2)(t_2 - t_1)^{2\beta}} \right). \]  
			It is easy to see that $||f_\ve||_{L^2(Q_T)}$ is uniformly bounded for $0 < \ve \leq \ve_0$.
			
			Hence, taking $\delta \to 0$ and using the statements we have derived regarding the left and right hand sides of \eqref{TestParts}, we see that
			\begin{equation*}
				\frac{M}{2}(t_2 - t_1)^\beta \frac{M^2}{16K_3^2}(t_2 - t_1)^{2\beta} \leq \frac{1}{(M^2/16K_3^2)(t_2 - t_1)^{2\beta}}||f_\ve||_{L^2(Q_T)} \frac{C\sqrt{2}M}{4K_3}(t_2 - t_1)^\beta (t_2 - t_1)^{1/2}.
			\end{equation*}
			This implies that
			\begin{equation*}
				M \leq \tilde{C}^{1/4},
			\end{equation*}
			where $\tilde{C}$ is a constant independent of $M$ and $\ve$.  This proves the lemma.
		\end{proof}
	\end{Lem}
	
	Because $h_\ve(\cdot,T) \in C^{1/2}_x(\bar{\Omega})$ and $h_\ve(x,\cdot) \in C^{1/8}_t[0,T_\t{loc}]$ for $x \in \Omega$ and $0 \leq T \leq T_\t{loc}$, \cite{E} (Theorem 9.3 on page 316) implies that $h_\ve$ can be extended as a solution to \eqref{KerchmanPert} on $Q_{T_\t{loc}}$.  Lemmas \ref{KSTEstLem}, \ref{KtHolderLem}, and \eqref{xHolder} imply that $\{ h_\ve \}_{\ve >0}$ is a uniformly bounded, equicontinuous family of functions on $Q_{T_\t{loc}}$.  Due to the Arzel\`{a}-Ascoli lemma, this will allow us to find weak solutions to \eqref{KerchmanAbs} in the next section in a natural sense.  Similarly, in the setting where $|\Omega| < \pi$, Lemmas \ref{PertGronwallBound} and \ref{KtHolderLem} and statement \eqref{xHolder} imply that $\left\{ h_\ve \right\}_{\ve > 0} $ is a uniformly bounded and equicontinuous family of functions on $Q_T$ for any finite time $T$.

%%%%%%%%%%%%%%%%%%%%%%%%%%%%%%%%%%%%%%%%%%%%%%%%%%%%%%%%%%%%%%%%%%%%%%%%%%%%%%%%%%%%%%%%%%%%
%%%%%%%%%%%%%%%%%%%%%%%%%%%%%%%%%%%%%%%%%%%%%%%%%%%%%%%%%%%%%%%%%%%%%%%%%%%%%%%%%%%%%%%%%%%%
	
	\section{Weak Solutions to Model I}\label{M1WeakSolutions}
	We now consider the initial boundary value problem 
	\begin{equation}\label{KerchmanAbs}
	\begin{cases}
	&h_t = - hh_x - S\left[ |h|^3(h_x + h_{xxx}) \right]_x \t{ in $Q_T$}, \\
	&h(x,0) = h_0(x) \in H^1(\Omega), \\
	& \partial_x^j h(-a,\cdot) = \partial_x^j h(a,\cdot) \t{ for $t \in (0,T)$}, \ j = 0,1,2,3.
	\end{cases}
	\end{equation}
	We define a weak solution to \eqref{KerchmanAbs} as follows:
	\begin{Def}\label{KWeakSol}
		Let $h$ be defined on $Q_T$ such that
		\begin{align}
			\label{ContProp}&h \in C_{x,t}^{1/2,1/8}(\overline{Q_T}) \cap L^\infty(0,T; H^1(\Omega)), \\
			\label{L2Prop}&h_t \in L^2(0,T; (H^1(\Omega))^*), \\
			\label{DerivProp}&h \in C_{x,t}^{4,1}(P), \\
			&|h|^{3/2}(h_{xxx} + h_x) \in L^2(P),
		\end{align}
		where $P = \overline{Q_T} \setminus (\{ (h= 0) \} \cup \{ t = 0 \})$.   Suppose that $h$ satisfies \eqref{KerchmanAbs} in the following sense:
		\begin{equation}\label{WeakSolProp}				
			\iint_{Q_T} h_t \phi \ dx \ dt - \iint_P \left( \frac{h^2}{2}  + S|h|^3\left( h_x + h_{xxx} \right) \right) \phi_x \ dx \ dt = 0
		\end{equation}
		for all $\phi \in C^1(Q_T)$ with $\phi(a,\cdot) = \phi(-a,\cdot)$.  Further
		\begin{align}
			\label{Tto0Prop}&h(\cdot,t) \to h(\cdot,0) \text{ pointwise and strongly in $L^2(\Omega)$ as $t \to 0$} \\
			\label{BoundaryProp}&\partial_x^j h(a,t) = \partial_x^j h(-a,t) \text{ for $j = 0,1,2,3$ on $(\partial \Omega \times (0,T)) \setminus \left( \{ h = 0 \} \cup \{ t= 0\} \right)$}.
		\end{align}  
		Then we call $h$ a \emph{weak solution} to the problem \eqref{KerchmanAbs}.
	\end{Def}
	
	Let $T$ be a uniform time of existence for a family of solutions $\{ h_\ve \}_{\ve > 0}$.  Because $\{ h_\ve \}_{\ve > 0}$ is a uniformly bounded and equicontinuous family of functions, then by the Arzel\`{a}-Ascoli lemma there is a subsequence $\ve_k \to 0$ such that
	\begin{equation}\label{SolConv}
		h_{\ve_k} \to h \ \ \ \text{uniformly in } \bar{Q}_{T}.
	\end{equation}
	Henceforth, we refer to this subsequence as $\ve \to 0$.
	
		\begin{Thm}\label{WeakSolThm}
		Let $\Omega$ be a bounded interval in $\R$. Any function $h$ obtained as in \eqref{SolConv} is a weak solution to \eqref{KerchmanAbs} in $Q_{T_{loc}}$, where $T_{loc} >0$ is given in Lemma \ref{KSTEstLem}.  If we further assume that $\Omega \subset \left( -\frac{\pi}{2},\frac{\pi}{2} \right)$, then $h$ is a weak solution to \eqref{KerchmanAbs} in $Q_T$, where $T>0$ can be taken to be arbitrarily large.
		\begin{proof}
			It is clear that \eqref{ContProp} follows by the fact that $h_\ve \to h$ uniformly in $Q_T$.  Now take $\phi \in \Lip \left( Q_T \right)$ such that $\phi = 0$ near $t = 0$ and $t = T$.  Then for each $0 < \ve \leq \ve_0$, we have
			\begin{align*}
				 &0 =  \iint_{Q_T} h_\ve \phi_t \ dx \, dt + \iint_{Q_T} \frac{h_\ve^2}{2} \phi_x \ dx \, dt   \\
				 & \hspace{.2cm} + S\iint_{Q_T} |h_\ve|^3 (h_{\ve,x} + h_{\ve,xxx} ) \phi_x \ dx \, dt + S\ve \iint_{Q_T} (h_{\ve,x} + h_{\ve,xxx}) \phi_x \ dx \, dt  .
			\end{align*}
			By \eqref{LongtimeUnifBound1}, Cauchy's inequality, and the Sobolev inequality, it follows that the expression $\ve^{1/2} ||h_{\ve,x} + h_{\ve,xxx}||_{L^2(Q_T)}$ is uniformly bounded with respect to $\ve$. Then, we see that
			\begin{align*}
				\ve \iint_{Q_T} |(h_{\ve,x} + h_{\ve,xxx}) \phi_x|\ dx \, dt &\leq \ve ||h_{\ve,x} + h_{\ve,xxx}||_{L^2(Q_T)} ||\phi_x||_{L^2(Q_T)}\\
				&\leq C \ve^{1/2},
			\end{align*}
			where $C$ is a constant independent of $\ve$.  Therefore
			\begin{equation*}
				\lim_{\ve \to 0} \ve \iint_{Q_T} (h_{\ve,x} + h_{\ve,xxx}) \phi_x \ dx \, dt = 0,
			\end{equation*}
			
			Note that our a priori estimates imply $H_\ve = (|h_\ve|^3 + \ve)^{1/2}(h_{\ve,x} + h_{\ve,xxx})$ is uniformly bounded in $L^2(Q_T)$, which in turn implies that $H_\ve \rightharpoonup H \in L^2(Q_T)$.  Regularity theory of uniformly parabolic equations and the fact that $h_\ve$ are uniformly H\"{o}lder continuous imply that
			\begin{equation}\label{DerivCon}
				h_{\ve,t}, h_{\ve,x}, h_{\ve,xx}, h_{\ve,xxx}, \text{ and } h_{\ve,xxxx} \text{ are uniformly convergent on any compact subset of $P$,}
			\end{equation}
			and hence \eqref{DerivProp} and \eqref{BoundaryProp}.  Furthermore, \eqref{DerivCon} and $H_\ve \rightharpoonup H$ in $L^2(Q_T)$ tells us that $H = |h|^3(h_x + h_{xxx})$ on $P$.
			
			Setting $f_\ve = \frac{h_\ve^2}{2} + S |h_\ve|^3(h_{\ve,x} + h_{\ve,xxx})$, we have for any $\delta > 0$
			\begin{equation}\label{fConv}
				\lim_{\ve \to 0} \iint_{|h|> \delta} f_\ve\phi_x \ dx \, dt = \iint_{|h| > \delta} f \phi_x \ dx \, dt.
			\end{equation}
			On the other hand, we can choose $0 < \ve \leq \ve_0$ small enough (dependent on $\delta$) so that
			\begin{align*}
				& \left| \iint_{|h| \leq \delta } f_\ve \phi_x \ dx \, dt \right| \leq \iint_{|h|\leq \delta } \frac{h_\ve^2}{2} |\phi_x| \ dx \, dt + S\iint_{|h| \leq \delta} |h_\ve|^3|(h_{\ve,x} + h_{\ve,xxx})\phi_x| \ dx \, dt \\
				& \hspace{0cm} \leq \iint_{|h|\leq \delta } \frac{h_\ve^2}{2} |\phi_x| \ dx \, dt + S\left( \iint_{|h|\leq \delta } |h_\ve|^3 \phi_x^2  dx  dt \right)^{1/2} \left( \iint_{|h| \leq \delta } |h_\ve|^3(h_{\ve,x} + h_{\ve,xxx})^2  dx  dt \right)^{1/2} \\
				& \hspace{2cm} \leq C_S \delta^{3/2},
			\end{align*}
			where $C_S$ is independent of $\delta$.  Combining this fact with \eqref{fConv} implies that
			\begin{equation}\label{fPConv}
				\lim_{\ve \to 0} \iint_{Q_T} f_\ve \phi_x \ dx \, dt = \iint_P f \phi_x \ dx \, dt,
			\end{equation}
			whence \eqref{WeakSolProp} follows.
		\end{proof}
	\end{Thm}

\section{Non-Negativity of Solutions to Model I}\label{M1Non-Negativity}
Using similar techniques as in \cite{BF}, we can give a non-negativity result for solutions constructed in Section \ref{M1WeakSolutions}:
	\begin{Thm}\label{Non-Neg}		
		Let $h$ be a weak solution to \eqref{KerchmanAbs} as constructed in Theorem \ref{KWeakSol} with $h_0 \geqslant 0$.  Further assume that $\int_\Omega h_0^{-1} \ dx < \infty$.  Then $h \geq 0$.  Furthermore, for each $T \in [0,\hat{T}]$, where $\hat{T}$ is a time of existence as constructed in section \ref{M1WeakSolutions}, the set $E_T = \{ x \in \Omega : \  h(x,T) = 0 \}$ is of measure zero.  Also, $\int_\Omega \frac{dx}{h(x,t)}$ is uniformly bounded.  Finally, if one further assumes that $\Omega \subset \left( -\frac{\pi}{2}, \frac{\pi}{2} \right)$, then $\int_\Omega G_\ve(h_\ve(x,T)) \ dx$ is a monotonically decreasing function on $T$.
		\begin{proof}
		Recall the definitions of $g_\ve$ and $G_\ve$ for $\ve > 0$:
			\begin{equation*}
				g_\ve(s) = -\int_s^A \frac{dr}{|r|^3 + \ve}, \quad G_\ve(s) = -\int_s^A g_\ve(r) \ dr,
			\end{equation*}
		where $A \geqslant \max |h_\ve|,$ which is a finite number by a priori estimates on $h_\ve$.  It follows by definition of $G_\ve$ that for $\ve > 0$ we have
			\begin{align*}
				\int_\Omega G_\ve(h_{0,\ve}(x)) &< C_1 + \int_\Omega h_{0,\ve}^{-1} \ dx \t{, where $C_1$ is independent of $\ve$} \\
				&\leq C_1 + \int_\Omega h_0^{-1} \ dx \t{, as $h_{0,\ve} \geq h_0$} \\
				&< \infty \t{, by hypothesis}
			\end{align*}
		and this bound is clearly independent of $\ve$.
			
		Multiplying \eqref{KerchmanPert} by $g_\ve(t)$ and integrating over $Q_T$, we see that
			\begin{align*}
				 & \int_\Omega G_\ve(h_\ve (x,T)) \ dx - \int_\Omega G_\ve(h_{0,\ve}(x)) \ dx = 	\iint_{Q_T} g_\ve(h) h_t \ dx \, dt \\
				&\hspace{1cm} = -\iint_{Q_T} \left[ \frac{h_\ve^2}{2} + S\left( |h_\ve|^3 + \ve \right)(h_{\ve,x} + h_{\ve,xxx}) \right]_x g_\ve(h) \ dx \, dt \\
				&\hspace{1cm}= -\iint_{Q_T} h_\ve g_\ve(h_\ve) h_{\ve,x} - S(h_{\ve,x} + h_{\ve,xxx}) h_{\ve,x} \ dx \, dt \\
				&\hspace{1cm}= -\iint_{Q_T} h_\ve h_{\ve,x} g_\ve(h_\ve) \ dx \ dt + S\iint_{Q_T} h_{\ve,x}^2 \ dx \ dt - S\iint_{Q_T} h_{\ve,xx}^2 \ dx \ dt,
			\end{align*}
		where the last two equalities follow by integrating by parts and using the fact that
		\[ \frac{\partial}{\partial x} g_\ve(h_\ve) = \frac{h_{\ve,x}}{|h_\ve|^3 + \ve}. \]  Hence, we have
			\begin{align*}
				& \int_\Omega G_\ve(h_\ve(x,T)) \ dx + S\iint_{Q_T} h_{\ve,xx}^2 \ dx \, dt  \\
				& \hspace{1cm} = -\iint_{Q_T} h_\ve g_\ve(h_\ve) h_{\ve,x} \ dx \ dt + S\iint_{Q_T} h_{\ve,x}^2 \ dx \, dt + \int_\Omega G_\ve(h_{0,\ve}(x)) \ dx. 
			\end{align*}
		We can define
			\begin{equation*}
				F_\ve(s) := \int_0^s r g_\ve(r) \ dr.
			\end{equation*}
		Using the fundamental theorem of calculus, it is clear that
			\begin{equation*}
				\int_\Omega h_\ve g_\ve (h_\ve) h_{\ve,x} \ dx = \int_\Omega \left[ F_\ve(h_\ve) \right]_x \ dx = F_\ve(h_\ve) \bigg|_{\partial \Omega}.
			\end{equation*}
		The periodic boundary conditions on $h_\ve$ implies that this integral is zero, and hence
			\begin{equation*}
				\int_{\Omega} G_\ve(h_\ve(x,T)) \ dx + S\iint_{Q_T} h_{\ve,xx}^2 \ dx \ dt = \int_{\Omega} G_\ve(h_{0,\ve}) \ dx + S\iint_{Q_T} h_{\ve,x}^2 \ dx \ dt.
			\end{equation*}
		If we have that $\int_\Omega h_{\ve,x}^2(x,t) \ dx$ is uniformly bounded for $0 \leq T \leq \hat{T}$, then it follows that $\int_\Omega G_\ve(h_\ve(x,T)) \ dx$ and $\iint_{Q_T} h_{\ve,xx}^2 \ dx \ dt$ are bounded for all $0 \leq T \leq \hat{T}$.
			
		If we further assume that $\Omega \subset \left( -\frac{\pi}{2}, \frac{\pi}{2} \right)$, we can apply the Poincar\'{e} inequality to $h_x$ with $\int_\Omega h_x = 0$ by the periodic boundary conditions.  As a result, we see that 
			\begin{equation*}
				\int_\Omega G_\ve(h_\ve(x,T)) \ dx + \left( 1 - \left( \frac{|\Omega|}{\pi} \right)^2 \right) \iint_{Q_T} h_{\ve,xx}^2 \leq \int_\Omega G_\ve(h_{0,\ve}) \ dx.
			\end{equation*}
		 From this inequality, we see that $\int_{\Omega} G_\ve(h_\ve(x,T)) \ dx$ is a decreasing function on $T$.
		
		Suppose, toward a contradiction, that there a point $(x_0,t_0) \in Q_{\hat{T}}$ so that $h(x_0,t_0) < 0$.  Because $h_\ve \to h$ uniformly in $Q_{\hat{T}}$, there is $\delta > 0$ and $\ve_0 > 0$ so that for every $0 < \ve \leq \ve_0$ and every $x \in \Omega$ satisfying $|x - x_0| < \delta$, we have $h_\ve(x,t_0) < -\delta$.  However, this implies
		\begin{align*}
			G_\ve(h_\ve(x,t_0)) = - \int_{h_\ve(x,t_0)}^A g_\ve(s) \ ds \geq - \int_{-\delta}^0 g_\ve(s) \ ds.
		\end{align*}
		Note that this lower bound tends to $-\int_{-\delta}^0 g_0(s) \ ds$ as $\ve \to 0$.  However, we have that $g_0(s) = \infty$ for $s < 0$, so the integral on the right is infinite.  This implies that
		\begin{equation*}
			\lim_{\ve \to 0} \int_\Omega G_\ve(h_\ve(\cdot,T)) \ dx = \infty,
		\end{equation*}
		which is a contradiction.  Hence, $h \geq 0$ in $Q_T$.
		
		Now, suppose toward a contradiction that there is a $t_0$ in $[0,\hat{T}]$ so that $\t{meas}(E_{t_0}) > 0$.  Then because $h_\ve \to h$ uniformly, there is a modulus of continuity $\sigma(\ve) > 0$ so that $h_\ve(x,t_0) < \sigma(\ve)$ for $x \in E_{t_0}$.  This implies that for $x \in E_{t_0}$ and $\delta > 0$, we have
		\begin{align*}
			G_\ve(h_\ve(x,t_0)) &= - \int_{h_\ve(x,t_0)}^A g_\ve(s) \ ds  \geq - \int_{\sigma(\ve)}^A g_\ve(s) \ ds \\
			&\geq - \int_\delta^A g_\ve(s) \ ds \to - \int_\delta^A g_0(s) \ ds
		\end{align*}
		provided that $\ve$ is taken small enough so that $\sigma(\ve) < \delta$.  It is also easy to show that
		\begin{equation*}
			-\int_\delta^A g_0(s) \ ds \geq c\delta^{-1},
		\end{equation*}
		where $A \geqslant \max |h_\ve|$, which is uniformly bounded for $0 \leq T \leq \hat{T}$.  This bound implies
		\begin{equation*}
			\int_\Omega G_\ve(h_\ve(x,t_0)) \ dx \geq c\delta^{-1} \text{meas}(E_{t_0})
		\end{equation*}
		which tends to infinty as $\delta$ (and hence $\ve$) go to zero.  This is a contradiction.
		
		Finally, note that by definition of $G_0(s)$, we have that for $(x,t)$ such that 
		\[
		h(x,t) > 0, \ \lim_{\ve \to 0} G_\ve(h_\ve(x,t)) = G_0(h(x,t)). \]  
		Because $E_t$ has measure zero for every $0 \leq t \leq \hat{T}$, it follows that this limit is valid for almost all $x$ in $\Omega$.  Then observe that uniform convergence of the integrand for positive $s$ yields
		\begin{align*}
			G_0(s) &= \lim_{\ve \to 0} G_\ve(s)	= \lim_{\ve \to 0} \int_s^A \int_r^A \frac{1}{\rho^3 + \ve} \ d\rho dr \\
			&= \int_s^A \int_r^A \frac{1}{\rho^3} \ d\rho dr = \frac{1}{2s} - \frac{1}{A} + \frac{s}{2A^2}.
		\end{align*}
		In particular, for $h(x,t) > 0$ we have
		\begin{equation*}
			G_0(h(x,t)) = \frac{1}{h(x,t)} - \frac{1}{A} + \frac{h(x,t)}{2A^2}.
		\end{equation*}
		Integrating over $\Omega$ yields
		\begin{equation*}
			\int_\Omega \frac{dx}{h(x,t)}  = \int_\Omega G_0(h(x,t)) \ dx + \frac{|\Omega|}{A} - \frac{M}{2A^2}.
		\end{equation*}
		Finally, an application of Fatou's lemma and the fact that the measure of $E_{T}$ is zero for each $T \in [0, \hat{T}]$ implies that $\int_\Omega \frac{dx}{h(x,t)}$ is uniformly bounded.
		\end{proof}
	\end{Thm}

	The non-negativity result gives a proof of the following theorem:
	\begin{Thm}\label{KerchWeakSol}
		Any function as obtained in \eqref{SolConv} is a weak solution to \eqref{Kerchman}.
	\end{Thm}

\section{Model II}\label{M2}

\subsection{Local in Time Theory}

We now discuss the model given by \eqref{Frenkel}.  As with Model I, \eqref{Frenkel} is degenerate if $h$ vanishes, so we must regularize the problem by analyzing
\begin{equation}\label{FrenkelPert}
\begin{cases}
&h_{\ve,t} =  -2h_\ve^2 h_{\ve,x} - S \left[ \left( |h_\ve|^3 + \ve \right)(h_{\ve,x} + h_{\ve,xxx}) \right]_x = 0 \t{ in $Q_T$}, \\
&h_\ve(x,0) = h_{0,\ve}(x) \in C^{4 + \gamma}(\Omega) \ \t{for some } \gamma \in (0,1), \\
& \partial_x^j h_\ve(-a,\cdot) = \partial_x^j h_\ve(a,\cdot) \text{ for $t \in (0,T), \ j = 0,1,2,3$}.
\end{cases}
\end{equation}
One can prove  local in time energy identities and estimates for Model I that are essentially identical to those proved for Model I.  Mirroring the work done in section \ref{M1EnergyEstimates} with \eqref{FrenkelPert}, we prove uniform a priori control of norms of $h_\ve$ in $H^1(\Omega)$, $L^\infty(\Omega)$, $C^{1/2}_x(\Omega)$ and $C^{1/8}_t[0,T_{\t{loc}}]$.  Theorem 6.3 \cite{E} tells us that for each $\ve > 0$ there is a solution $h_\ve$ to \eqref{FrenkelPert} on $Q_{\tau_\ve}$, where $\tau_\ve > 0$.  The a priori control listed above allows us to apply Theorem 9.3 (p. 316) and Corollary 2 (p. 213) \cite{E} in order to extend each solution $h_\ve$ to $Q_{T_\t{loc}}$.  

As in section \ref{M1WeakSolutions}, we then use the uniform boundedness and H\"{o}lder continuity in order to apply the Arzel\`{a}-Ascoli lemma as we take $\ve$ to zero.  Then writing the problem
\begin{equation}\label{FrenkelAbs}
\begin{cases}
&h_t = -2h^2h_x - S \left[ |h|^3(h_x + h_{xxx}) \right]_x = 0 \t{ in $Q_{T_\t{loc}}$}, \\
&h(x,0) = h_0(x) \in H^1(\Omega), \\
& \partial_x^j h(-a,\cdot) = \partial_x^j h(a,\cdot) \ \text{for $t \in (0,T_\t{loc}), \ j = 0,1,2,3$},
\end{cases}
\end{equation}
and a definition comparable to Definition \ref{KWeakSol}, we prove that the limit as $\ve$ to zero (along a subsequence) satisfies such a definition.  Finally, we move forward to prove that this limit is also non-negative as in section \ref{M1Non-Negativity}, which proves that it is a weak solution to \eqref{Frenkel}.

\subsection{Global in Time Estimates on $\Omega \subset \left( -\frac{\pi}{2},\frac{\pi}{2} \right) $}\label{M2Global}

In this section, we assume that $\Omega \subset \left( -\frac{\pi}{2}, \frac{\pi}{2} \right)$. We use non-negativity of solutions to \eqref{Frenkel} and \eqref{FrenkelPert} in order to write them, respectively, as
\begin{equation}\label{FrenkelGlobalPos}
\begin{cases}
&h_t + S[ h^3(h + h_{xx} + \frac{2}{3S}x)_x]_x = 0 \text{ in $Q_T$}, \\
&h(x,0) = h_0(x) \in H^1(\Omega), \\
&\partial_x^j h(-a,\cdot) = \partial_x^j h(a,\cdot) \text{ for $t \in (0,T), \ j = 0,1,2,3$}, \\
& \frac{2}{3S}x \left( h + h_{xx} + \frac{2}{3S}x \right)_x \Big|_{-a}^a \equiv 0 ,
\end{cases}
\end{equation}
and
\begin{equation}\label{FrenkelGlobalPosPert}
\begin{cases}
&h_{\ve,t} + S[ (h_\ve^3 + \ve)(h_\ve + h_{\ve,xx} + \frac{2}{3S}x)_x]_x = 0 \text{ in $Q_T$},  \\
&h_\ve(x,0) = h_{0,\ve}(x) \in C^{4 + \gamma}(\Omega), \\
&\partial_x^j h_\ve(-a,\cdot) = \partial_x^j h_\ve(a,\cdot) \text{ for $t \in (0,T), \ j = 0,1,2,3$}, \\
& \frac{2}{3S}x \left( h_\ve + h_{\ve,xx} + \frac{2}{3S}x \right)_x \Big|_{-a}^a \equiv 0 ,
\end{cases}
\end{equation}
where the boundary conditions have been added. Note that it follows by a similar argument as in section \ref{M1Non-Negativity} that for sufficiently small $\ve > 0$, we must have $h_\ve \geq 0$ on $Q_{T_\ve}$, legitimizing \eqref{FrenkelGlobalPosPert}.  Now, we provide a uniform $H^1(\Omega)$ bound on $h_\ve(\cdot,T)$, independent of $\ve>0$ and $T> 0$.

\begin{Lem}\label{FGlobalEst}
	Suppose $h_\ve$ is a solution to \eqref{FrenkelGlobalPosPert}.  Then $||h_\ve(\cdot,T)||_{H^1(\Omega)}$ is uniformly bounded for all $T > 0$ and $\ve > 0$ sufficiently small.
	\begin{proof}
		Suppose $h := h_\ve$ is a solution to \eqref{FrenkelGlobalPosPert}. Then multiplying \eqref{FrenkelGlobalPosPert} by $\left(h + h_{xx} + \frac{2}{3S}x \right)$ and integrating over $\Omega$ yields
		\begin{equation*}
		0 = \int_\Omega \left[ hh_t+ h_{xx}h_t + \frac{2}{3S}xh_t + S\left[ \left( h^3 + \ve \right) \left( h + h_{xx} + \frac{2}{3S}x  \right)_x \right]_x \left( h + h_{xx} + \frac{2}{3S}x \right) \right] dx.
		\end{equation*}
		Integrating by parts and using the boundary conditions prescribed in \eqref{FrenkelGlobalPosPert}, we obtain
		\begin{equation*}
		0 = \frac{1}{2} \frac{d}{dt}  \int_\Omega \left( h_x^2 -h^2  - \frac{4}{3S} x h \right) dx + S \int_\Omega \left( h^3 + \ve \right) \left( h + h_{xx} + \frac{2}{3S}x \right)_x^2  dx .
		\end{equation*}
		Integrating in time from $0$ to $T$ and applying the fundamental theorem of calculus, it follows that
		\begin{equation}\label{GlobalEnergyIneq1}
		\tilde{\mc{E}}(h(\cdot,T)) + S\iint_{Q_T} \left( h^3 + \ve \right) \left[ \left( h + h_{xx} + \frac{2}{3S}x \right)_x \right]^2 dx = \tilde{\mc{E}}(h_{0\ve}),
		\end{equation}
		where $\tilde{\mc{E}}(h) := \frac{1}{2}\int_\Omega (h_x^2 - h^2 - \frac{4}{3S}x h) \ dx$.  From \eqref{GlobalEnergyIneq1}, it follows that
		\begin{equation*}
		\int_\Omega h_x^2(x,T) \ dx \leq 2\tilde{\mc{E}}(h_{0,\ve}) + \int_\Omega h^2(x,t) \ dx + \frac{4}{3S} \int_\Omega x h(x,t) \ dx.
		\end{equation*}
		Using integration parts with periodic boundary conditions and applying the Cauchy-Schwarz inequality to the right-hand integral, one obtains
		\begin{equation}\label{GlobalEnergyIneq2}
		\int_\Omega h_x^2 (x,T) \ dx \leq 2 \tilde{\mc{E}}(h_{0,\ve}) + \int_\Omega h^2(x,T) \ dx + \frac{2}{3S} \left( \int_\Omega x^4 \ dx \right)^{1/2} \left( \int_\Omega h^2 \ dx \right)^{1/2}.
		\end{equation}
		Recalling the Poincar\'{e} inequality, we have
		\begin{equation}\label{Poincare}
		\int_\Omega h^2 \ dx \leq \left( \frac{|\Omega|}{\pi} \right)^2 \int_\Omega h_x^2 \ dx + \frac{M_\ve^2}{|\Omega|}.
		\end{equation}
		Applying \eqref{Poincare} to \eqref{GlobalEnergyIneq2} and integrating, we find that
		\begin{equation}\label{GlobalEnergyIneq3}
		\left[ 1 - \left( \frac{|\Omega|}{\pi} \right)^2 \right] \int_\Omega h_x^2(x,T) \leq 2\tilde{\mc{E}}(h_{0,\ve}) + \frac{M_\ve^2}{|\Omega|} + \frac{2}{3S} \left( \frac{2a^5}{5} \right)^{1/2} \left( \int_\Omega h_x^2(x,T) \ dx \right)^{1/2}.
		\end{equation}
		An application of Cauchy's inequality yields
		\begin{equation*}
		\left[ 1 - \left( \frac{|\Omega|}{\pi} \right)^2 \right] \int_\Omega h_x^2(x,T) \ dx \leq \delta \int_\Omega h_x^2(x,T) \ dx +2\tilde{\mc{E}}(h_{0,\ve}) + \frac{M_\ve^2}{|\Omega|} + C(\delta),
		\end{equation*}
		where we can choose $\delta = \frac{1}{2}\left[ 1 - \left( \frac{|\Omega|}{\pi} \right)^2 \right] > 0$.  It follows that
		\begin{equation}\label{GlobalEnergyIneq4}
		\int_\Omega h_x^2(x,T) \ dx \leq \delta^{-1} \left[ 2\tilde{\mc{E}}(h_{0,\ve}) + \frac{M_\ve^2}{|\Omega|} + C(\delta) \right] =: A,
		\end{equation}
		i.e. $\int_\Omega h_{\ve,x}^2(x,T) \ dx$ is uniformly bounded, independent of $T$ and $\ve$.  The Poincar\'{e} inequality again implies that $\int_\Omega h_\ve^2(x,T) \ dx$ is uniformly bounded so that $||h_\ve(\cdot,T)||_{H^1(\Omega)}$ is uniformly bounded.  The Sobolev embedding theorem yields a uniform bound on $||h_\ve||_{L^\infty(Q_\infty)},$ where $Q_\infty = \Omega \times (0, \infty)$.
	\end{proof}
\end{Lem}

One can use the arguments in section \ref{M1Holder} to show that $h_\ve \in C_{x,t}^{1/2,1/8}(Q_\infty)$ uniformly.  Again, it follows by the arguments in sections \ref{M1WeakSolutions} and \ref{M1Non-Negativity} that taking $\ve$ to zero yields a non-negative weak solution to \eqref{FrenkelGlobalPos} in $Q_\infty$.

\section{Extending the Results and Future Work}\label{FutureWork}

One can also consider a general drift term of the form $(|h|^{\lambda-1}h)_x.$  Using the same arguments as in Section \ref{M1ShortTimeEstimates}, one can obtain local existence theory for any $\lambda  \geq 1$.  Furthermore, one can realize a global existence theory for $1 \leq \lambda \leq \frac{5}{2}$ on domains $\Omega \subset \left( -\frac{\pi}{2},\frac{\pi}{2} \right).$  It is for this reason that we require the nonlinear boundary condition in Section \ref{M2Global}.  However, in this paper, we only explore models corresponding to concrete physical phenomena.

An immediate extension of the results for Model I and Model II is local in time existence of weak solutions to the problem represented by the mixed model:
\begin{equation}\label{Mixed}
\begin{cases}
&h_t = - \alpha hh_x - (1- \alpha)h^2h_x - S\left[ h^3(h_x + h_{xxx}) \right]_x \t{ in $Q_T$},\\
&h(x,0) = h_0(x) \in H^1(\Omega), \\
& \partial_x^j h(-a,\cdot) = \partial_x^j h(a,\cdot) \t{ for $t \in (0,T)$}, \ j = 0,1,2,3, 
\end{cases}
\end{equation}
where $\alpha \in [0,1]$.  How to apply the local in time theory is very straightforward.  Similarly, the work of section \ref{M1LongTimeEstimates} should provide long-time estimates for a regularization of \eqref{Mixed}.  Following the steps set forth in sections \ref{M1WeakSolutions} and \ref{M1Non-Negativity} provide the existence of weak solutions to \eqref{Mixed}.

As discussed in the introduction, both Model I and Model II are limits of long-wave models discussed, from both the experimental and numerical standpoints, in \cite{camassa2012ring,ogrosky2013modeling,COO,camassa2017viscous,CMOV}.  One of such models is given below
	\begin{equation}\label{CMOV-Rcoord}
		\begin{cases}
			& \mu R_t = \rho g f_1(R;b)R_x + \frac{\gamma}{16R} \left[ f_2(R;b)(R_x + R^2 R_{xxx}) \right]_x \t{, in $Q_T = \Omega \times (0,T)$} \\
			& R(x,0) = R_0(x) \in H^1(\Omega) \\
			& \partial_x^j R(a,t) = \partial_x^j R(-a,t) \t{, for $t \in (0,T)$ and $j = 0,1,2,3,$},
		\end{cases}
	\end{equation}
where $\mu, \rho, g$ are all parameters of interest, $b$ is the radius of the cylinder, and $f_1$ and $f_2$ are functions given by
	\begin{align}
		f_1(R;b) &= \frac{1}{2} \left[ R^2 - b^2 - 2R^2 \log \left( \frac{R}{b} \right) \right] \\
		f_2(R;b) &= -\frac{b^4}{R^2} + 4b^2 - 3R^2 + 4R^2 \log \left( \frac{R}{b} \right).
	\end{align}
The next step is to apply the local in time energy methods used in section \ref{M1EnergyEstimates} to the corresponding long-wave models.  Because the degeneracies in the models in \cite{CMOV} are more complicated than the simply polynomials in \eqref{Kerchman} and \eqref{Frenkel}, the regularizations must be defined more meticulously. Defining weak solutions appropriately and adapting the non-negativity arguments of section \ref{M1Non-Negativity} will demonstrate the existence of weak solutions to the models.

%%%%%%%%%%%%%%%%%%%%%%%%%%%%%%%%%%%%%%%%%%%%%%%%%%%%%%%%%%%%%%%%%%%%%%%%%%%%%%%%%%%%%%%%%%%%	
\vspace{.5cm}

\noindent{\bf Acknowledgements:}  The authors thank Roberto Camassa and Reed Ogrosky for many helpful conversations about this family of thin film models. JLM and SS were supported in part by the NSF through JLM's NSF CAREER Grant DMS-1352353.

\bibliographystyle{amsalpha}
\bibliography{../../ThesisReferences}

\end{document}